\documentclass[12pt,a4paper]{amsart}

\usepackage{amsmath}
\usepackage{amssymb}
\usepackage{amsthm}
\usepackage{latexsym}
\usepackage{pstricks,multido}

\usepackage{amscd}
\usepackage{graphicx}

\theoremstyle{plain}
\newtheorem{theorem}{Theorem}[section]
\newtheorem{corollary}[theorem]{Corollary}
\newtheorem{lemma}[theorem]{Lemma}

\theoremstyle{definition}
\newtheorem{definition}{Definition}[section]

\theoremstyle{remark}

\newtheorem{example}{Example}[section]

\numberwithin{equation}{section}

\numberwithin{table}{section}

\numberwithin{figure}{section}

\title[On Noncrossing and nonnesting partitions of type $D$]
{On Noncrossing and nonnesting partitions of type $D$}

\author{Alessandro Conflitti}
\address{CMUC \\ Centro de Matem\'atica \\ Universidade de Coimbra \\ Apartado 3008 \\
3001-454 Coimbra \\ Portugal} \email{conflitt@mat.uc.pt}
\author{Ricardo Mamede}
\address{CMUC \\ Centro de Matem\'atica \\ Universidade de Coimbra \\ Apartado 3008 \\
3001-454 Coimbra\\ Portugal} \email{mamede@mat.uc.pt}



\thanks{Both authors were supported by CMUC - Centro de Matem\'atica da
Universidade Coimbra. The first author was also supported by FCT
Portuguese Foundation of Science and Technology (Funda\c{c}\~{a}o para a
Ci\^{e}ncia e a Tecnologia) Grant SFRH/BPD/30471/2006.}

\begin{document}

\begin{abstract}
We present an explicit bijection between noncrossing and nonnesting partitions of Coxeter systems of type
$D$ which preserves openers, closers and transients.
\end{abstract}

\maketitle

\section{Overview}

The lattice of set partitions of a set of $n$ elements can be interpreted as the intersection lattice for the hyperplane arragement corresponding to a  root system of type $A_{n-1}$, i.e. the symmetric group of $n$ objects, $S_{n}$. In particular, two of its subposets are very well--behaved and widely studied, i.e. the lattice of, respectively, noncrossing and nonnesting partitions, which have a lot of interesting combinatorial properties, see e.g. \cite{Bia,CDDSY,DZ,Edel,KZ,Kre,LYHC,NS,Sim1,Sim2,Sim3,Spe1,Spe2} and the references therein.

Recently, Victor Reiner \cite{reiner} and Alexander Postnikov, see Christos Athanasiadis' paper~\cite[\S 6, Remark 2]{atha} generalized the notion
of, respectively, noncrossing and nonnesting partitions to all classical reflection groups, and a new very active research area sprung up, namely generalizing previous known results held for noncrossing and nonnesting partitions of type $A$ to their type $B$ and $D$ analogue,
see e.g. \cite{Arm,atha1,Kra,Stump,Th} and the references therein.

In particular, it is well known that in type $A$ the number of noncrossing partition equals the number of nonnesting partition, and several  bijections have been established. Recently one of us in \cite{m,m1} solved the corresponding problem for the type $B$ (and type $C$) presenting an explicit bijection which falls back to one already known when restricted to the type $A$.

The goal of this paper is to present an explicit bijection for the type $D$, therefore solving the problem for all classical Weyl groups, and to show that this bijection has interesting additional combinatorial properties. For instance, for any classical reflection group and any partition it is possible to  define three different maps, called \emph{openers}, \emph{closers}, and \emph{transients}, going from the partition itself to subsets of the involved reflection group, and we show that our bijection preserves all these functions.
Furthermore, our map remains a bijection if restricted to the two sets of respectively, noncrossing and nonnesting partitions which are \emph{both} of type $B$ and type $D$, and therefore it is also a bijection between the partitions which are noncrossing in type $D$ but which are actually crossing in type $B$, and the partition which are nonnesting in type $D$ but which are actually nesting in type $B$.

Very recently, two other bijections have been presented for solving the
type $D$ problem, see~\cite{fink, stump1}, but both of them have different
designs and settings. Moreover in~\cite{fink} openers, closers and
transients are not preserved.


\section{Definitions and Preliminars}

In this section we review the notions of noncrossing and nonnesting
partitions of types $B$ and $D$, following closely \cite{atha,atha1},
and referring to \cite{BB,hum} for any undefined terminology and comprehensive references on Coxeter groups.
Throughout the paper, let $$[n]:=\{1,2,\ldots,n\},$$
$$[\pm n]:=\{\overline{1},\overline{2},\ldots,\overline{n},1,2,\ldots,n\}$$
for any positive integer $n$, where we set $\overline{i}:=-i$.

The Coxeter group $W^{B_n}$ of type  $B_n$ is realized combinatorially as the
hyperoctahedral group of signed permutations of $[\pm n]$. These
are permutations of $[\pm n]$ which commute with
the involution $i\mapsto \overline{i}$. We write the elements of
$W$ in cycle notation, using commas between elements.
The simple generators of $W^{B_n}$ are the transposition $(\overline{1},\; 1)$ and
the pairs $(\overline{i+1},\; \overline{i})(i,\; i+1)$ for
$i=1,\ldots,n-1$. The reflections in $W^{B_n}$ are the transpositions
$(\overline{i},\; i)$ for $i=1,\ldots,n$, and the pairs of
transpositions $(i,\; j)(\overline{j},\; \overline{i})$  for $i\neq j$.
Identifying the sets $[\pm n]$ and $[2n]$ through the map  $i\mapsto i$
for $i\in[n]$ and $i\mapsto n+\overline{i}$ for $i\in\{\overline{1},\overline{2},\ldots,\overline{n}\}$,
allows us to identify the hyperoctahedral group $W^{B_n}$ with the subgroup $U$ of $\mathfrak{S}_{2n}$
which commutes with  the permutation  $(1,n+1)(2,n+2)\cdots (n,2n)$.

Denoting by $e_1,\ldots,e_n$ the standard basis of $\mathbb{R}^n$, the
root system of type $B_n$ consists on the set of $2n^2$ vectors
\begin{displaymath}\Phi=\{\pm e_i:1\leq i\leq n\}\cup\{\pm e_i\pm e_j:i\neq j, 1\leq i,j\leq n\},\end{displaymath}
and we take
\begin{displaymath}\Phi^+= \{e_i:1\leq i\leq n\}\cup\{e_i\pm e_j:1\leq j<i\leq n\}\end{displaymath}
as a choice of positive roots. Each root $e_i,e_i-e_j$ and
$e_i+e_j$ defines a reflection that acts on  $\mathbb{R}^n$ as the permutation $(i,\;-i)$,
$(i,\;j)(-i,\;-j)$ and $(i,\;-j)(-i,\;j)$, respectively.

The crystallographic root system of type $D_n$ has $n(n - 1)$ positive roots, consisting of
$$\{e_i\pm e_j\},\text{ for }1\leq i<j\leq n.$$
Thus, $D_n$ is a sub--root system of $B_n$, and the Weyl group $W^{D_n}$ is an
(index 2) subgroup of the signed permutation group $W^{B_n}$, generated by the elements
$(1,\; \overline{2})(\overline{1},\; 2)$ and $(\overline{i+1},\; \overline{i})(i,\; i+1)$ for
$i=1,\ldots,n-1$. The reflections in $W^{D_n}$ are the transpositions
$(i,\; j)(\overline{i},\; \overline{j})$ for $1 \leq i < |j| \leq n$.
Any element of $W^{D_n}$ can be expressed uniquely (up to reordering) as a product of disjoint cycles
$$c_1\overline{c}_1\cdots c_k\overline{c}_kd_1\cdots d_r,$$
each having at least two elements, where $\overline{c}$ is obtained by negating the elements of $c$, $\overline{d}_j=d_j$ for $j=1,2,\ldots,r$, and $r$ is even (see \cite{brady}).

\subsection{Set partitions and partitions of type $B$ and $D$}

A partition of a finite set $S$ is a collection of pairwise disjoint, non-empty, subsets of $S$, called blocks, whose union is the whole set. A generic
set partition with no further restriction is sometimes denoted as partition of type $A$, because the lattice of all set partitions of a set of $n$ elements
can be interpreted as the intersection lattice for the hyperplane arragement corresponding to a root system of type $A_{n-1}$, i.e. the symmetric group of $n$ objects, $S_{n}$. A partition of $[n]$ can be visualized by placing the numbers $1,2,\ldots,n$ in this order along a line and then joining consecutive elements of each block by an arc. A singleton of a set partition is a block which has only one element, so it corresponds to an isolated vertex in the graphical representation. The  smallest element of each block is called an opener, the greatest is said a closer, and the remaining ones are called transients. These elements can be recognized by looking at the graphical representation of the partition: the openers correspond to singletons and to vertices from which an arc begins and no arc ends, the closers correspond to singletons and to vertices to which an arc ends and no arc begins, and the transients are the vertices from which an arc ends and another one begins.

A $B_n$--partition $\pi$ is a partition of $[\pm n]$ which has at most one block (called the zero block) fixed by negation and it is such that for any block $B$ of $\pi$, the set $-B$, obtained by negating the elements of $B$,
is also a block of $\pi$. The type of a $B_n$--partition $\pi$ is the integer partition whose parts are the cardinalities of the blocks of $\pi$, including one part for each pair of nonzero blocks $B,-B$.
A $D_n$--partition is a $B_n$--partition $\pi$ with the additional property that the zero block, when it is eventually present, has more than two elements.
The set of all $B_n$--partitions, ordered by refinement, is denoted by $\prod^B(n)$, and its subposet consisting of all $D_n$--partitions is denoted by
$\prod^D(n)$. The posets $\prod^B(n)$ and $\prod^D(n)$ are geometric lattices which are isomorphic to the intersection lattice of the $B_n$ and $D_n$ Coxeter hyperplane arrangement, respectively.

Identifying the sets $[\pm n]$ and $[2n]$ as before, we may represent $B_n$ and $D_n$--partitions graphically using the conventions made for their type
$A$ analogue, placing the integers $$\overline{1},\overline{2},\ldots,\overline{n},1,2,\ldots,n$$ along a line instead of the usual
$1,2,\ldots,2n$. For any $\pi\in\prod^B(n)$, let the set of openers $op(\pi)$ be made by the least element of all blocks of $\pi$ having only positive  integers; let the set of closers $cl(\pi)$ be made by the greatest element of all blocks of $\pi$ having only positive integers and by the absolute values  of the least and greatest elements of all blocks having positive and negative integers; and finally let the set of transients $tr(\pi)$ be made by all  elements of $[n]$ which are not in $op(\pi) \cup cl(\pi)$.

\begin{example}\label{ncex1}
The  $B_5$ set partition $\pi=\{\{-1,1\},\{2,3,5\},\{-2,-3,-5\},\{4\},\{-4\}\},$ represented below, has
type $(3,2,1)$, set of openers $op(\pi)=\{2,4\}$, closers $cl(\pi)=\{1,4,5\}$ and transients $tr(\pi)=\{3\}$. The openers, closers and transients can be visualized as in type $A$ by looking only at the arcs on the positive half of the representation of $\pi$. Note that $\pi$ is not a $D_n$-partition since the zero block has only two elements.

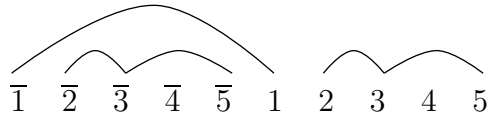
\begin{pspicture}(-3.4,-0.5)(3.4,2)
\multido{\i=1+1}{5}{$\;\;\overline{\i}\;\;$}\multido{\i=1+1}{5}{$\;\;\i\;\;$}
\pscurve[linewidth=0.5pt]{}(-6.5,0.5)(-4.6,1.4)(-3.05,0.5)
\pscurve[linewidth=0.5pt]{}(-5.8,0.5)(-5.4,0.8)(-5,0.5)
\pscurve[linewidth=0.5pt]{}(-5,0.5)(-4.3,0.8)(-3.6,0.5)

\pscurve[linewidth=0.5pt]{}(-2.4,0.5)(-2,0.8)(-1.6,0.5)
\pscurve[linewidth=0.5pt]{}(-1.6,0.5)(-0.9,0.8)(-0.3,0.5)
\end{pspicture}

\end{example}

\subsection{Noncrossing partitions}

Following \cite{atha1}, we now present the notion of noncrossing partitions of types $B$ and $D$.

Label the vertices of a convex $2n$--gon as $1,2,\ldots,n,\overline{1},\overline{2},\ldots,\overline{n}$ clockwise, in this order. Given a $B_n$--partition $\pi$ and a block $B$, let $\rho(B)$ be the convex hull of the set of vertices labeled with the elements of $B$. We say that $\pi$ is {\it noncrossing} if $\rho(B)$ and $\rho(B^{\prime})$ have empty intersection for any two distinct blocks $B$ and $B^{\prime}$ of $\pi$. Cutting the $2n$--gon between the integers $\overline{1}$ and $n$ and stretching it along a line, we get a graphical representation of the noncrossing partition $\pi$ where no two arcs cross.
The set of all noncrossing partitions of $[\pm n]$, denoted by $NC^{B}(n)$, is a subposet of $\prod^B(n)$ which is a self--dual and graded lattice
of rank $n$, see \cite{reiner}.

In the following example two noncrossing partitions are depicted.

\begin{example}
The $B_n$-partitions $\pi=\{\{\overline{4},4\},\{\overline{5},1,3\},\{5,\overline{1},\overline{3}\},\{2\},\{\overline{2}\}\}$ and
$\pi'=\{\{\overline{5},1\},$ $\{5,\overline{1}\},\{\overline{4},\overline{3},2\},\{3,4,\overline{2}\}\}$, represented below, are elements of $NC^B(5)$, with $op(\pi)=\{2\}, cl(\pi)=\{2,3,4,5\}, tr(\pi)=\{1\}$, and $op(\pi')=\emptyset, cl(\pi)=\{1,2,4,5\}, tr(\pi')=\{3\}$.

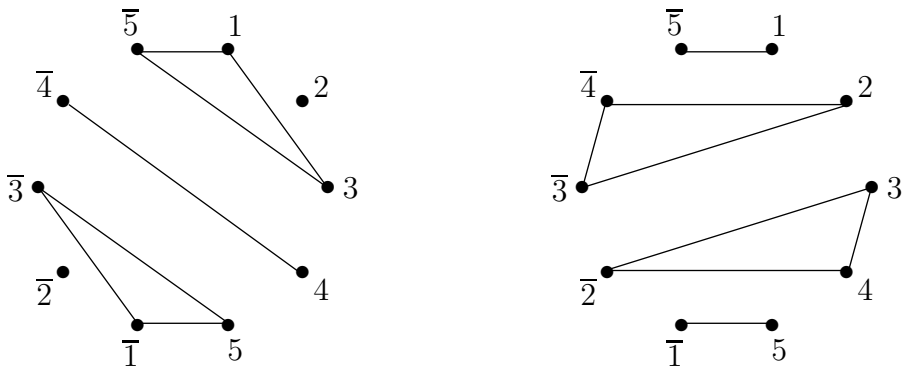
\begin{tabular}{cc}
 \begin{pspicture}(-3.4,-2.5)(3.4,3)
 \degrees[1]
 \multido{\n=0.0+.1}{10}{%
 \uput{1.8}[\n](0,0){$\bullet$}}
 \multido{\n=0.2+-0.1,\i=1+1}{5}{%
 \uput{2.1}[\n](0,0){\i}}
 \multido{\n=0.7+-0.1,\i=1+1}{5}{%
 \uput{2.1}[\n](0,0){$\overline{\i}$}}
 \psline[linewidth=0.5pt]{}(-1.5,1.1)(1.5,-1.1)
 \psline[linewidth=0.5pt]{}(-0.6,1.8)(0.6,1.8)(1.9,0)(-0.6,1.8)
 \psline[linewidth=0.5pt]{}(0.6,-1.8)(-0.6,-1.8)(-1.9,0)(0.6,-1.8)
 \end{pspicture}
&
 \begin{pspicture}(-3.4,-2.5)(3.4,3)
 \degrees[1]
 \multido{\n=0.0+.1}{10}{%
 \uput{1.8}[\n](0,0){$\bullet$}}
 \multido{\n=0.2+-0.1,\i=1+1}{5}{%
 \uput{2.1}[\n](0,0){\i}}
 \multido{\n=0.7+-0.1,\i=1+1}{5}{%
 \uput{2.1}[\n](0,0){$\overline{\i}$}}
 \psline[linewidth=0.5pt]{}(1.9,0)(-1.6,-1.1)(1.6,-1.1)(1.9,0)
 \psline[linewidth=0.5pt]{}(-1.9,0)(1.6,1.1)(-1.6,1.1)(-1.9,0)
 \psline[linewidth=0.5pt]{}(-0.6,1.8)(0.6,1.8)
 \psline[linewidth=0.5pt]{}(0.6,-1.8)(-0.6,-1.8)
 \end{pspicture}
\end{tabular}

\end{example}

Consider now the type $D$ case. Let us label the vertices of a regular $(2n-2)$--gon as $2,3,\ldots,n,\overline{2},\overline{3,}\ldots,\overline{n}$ clockwise, in this order, and label its centroid with both $1$ and $\overline{1}$. Given a $D_n$--partition $\pi$ and a block $B$ of $\pi$, let $\rho(B)$ be the convex hull of the set of vertices labeled with the elements of $B$. Two distinct blocks $B$ and $B^{\prime}$ of $\pi$ are said to cross if $\rho(B)$ and $\rho(B^{\prime})$ do not coincide and one of them contains a point of the other in its relative interior. Note that the case $\rho(B)=\rho(B^{\prime})$ can occur only when $B$ and $B^{\prime}$ are the singletons $\{1\}$ and $\{\overline{1}\}$, and that if $\pi$ has a zero block $B$, then $B$ and the block containing $1$ cross unless $\{1,\overline{1}\}\subseteq B$.

The poset $NC^D(n)$ is defined as the subposet of $\prod^D(n)$ consisting of those $D_n$--partitions $\pi$ with the property that no two blocks of $\pi$ cross. It is a graded lattice of rank $n$. Note that the zero block of $\pi\in NC^D(n)$, if it is eventually present, contains necessarily the
integers $1,\overline{1}$, and at least one more pair $i,\overline{i}$, with $i\neq \pm 1$.

\begin{example}
Below are graphical representations of the noncrossing $D_n$-partitions $\pi=\{\{\pm1,\pm3,\pm4\},\{\overline{5},2\}\{5,\overline{2}\}\}$ and $\pi'=\{\{\overline{1},4,5\},\{1,\overline{4},\overline{5}\},\{2,3\},\{\overline{2},\overline{3}\}\}$. Observe that while $\pi$ is not an element of $NC^B(5)$, $\pi'$ is. This may be checked by moving the integers $1$ and $\overline{1}$ to their respective places in the $2n$-gon and see if some cross occur.

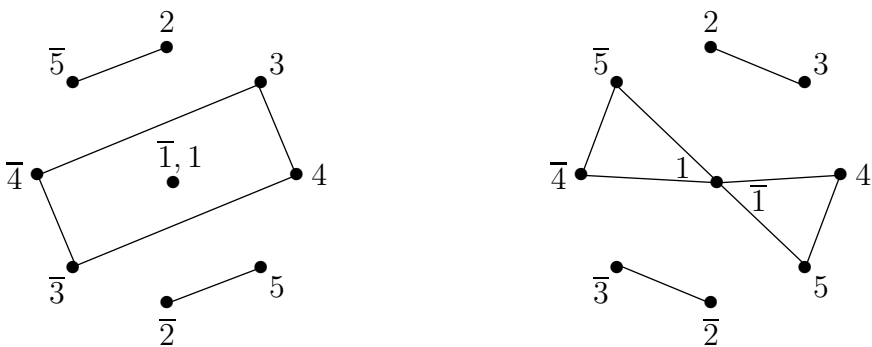
\begin{tabular}{cc}
 \begin{pspicture}(-3.4,-2.5)(3.4,2.5)
 \degrees[1]
 \multido{\n=0.0+.125}{8}{%
 \uput{1.6}[\n](0,0){$\bullet$}}
 \multido{\n=0.250+-0.125,\i=2+1}{4}{%
 \uput{1.9}[\n](0,0){\i}}
 \multido{\n=0.750+-0.125,\i=2+1}{4}{%
 \uput{1.9}[\n](0,0){$\overline{\i}$}}
 \uput[90](-0.3,0.25){$\overline{1},1$}
 \uput[0](-0.2,-0.1){$\bullet$}
 \psline[linewidth=0.5pt]{}(-1.7,0)(1.2,1.2)(1.7,0)(-1.2,-1.2)(-1.7,0)
 \psline[linewidth=0.5pt]{}(-1.3,1.2)(0,1.7)
 \psline[linewidth=0.5pt]{}(1.3,-1.2)(0,-1.7)
 \end{pspicture}
&
 \begin{pspicture}(-3.4,-2.5)(3.4,2.5)
 \degrees[1]
 \multido{\n=0.0+.125}{8}{%
 \uput{1.6}[\n](0,0){$\bullet$}}
 \multido{\n=0.250+-0.125,\i=2+1}{4}{%
 \uput{1.9}[\n](0,0){\i}}
 \multido{\n=0.750+-0.125,\i=2+1}{4}{%
 \uput{1.9}[\n](0,0){$\overline{\i}$}}
 \uput[90](-0.65,0.1){$1$}\uput[90](0.35,-0.3){$\overline{1}$}
 \uput[0](-0.2,-0.1){$\bullet$}
 \psline[linewidth=0.5pt]{}(-1.25,1.2)(0.1,-0.1)(-1.7,0)(-1.25,1.2)
 \psline[linewidth=0.5pt]{}(1.25,-1.2)(0.1,-0.1)(1.7,0)(1.25,-1.2)
 \psline[linewidth=0.5pt]{}(0,1.7)(1.2,1.2)
 \psline[linewidth=0.5pt]{}(0,-1.7)(-1.2,-1.2)
 \end{pspicture}
 \end{tabular}

We have $op(\pi)=\emptyset, cl(\pi)=\{2,4,5\}, tr(\pi)=\{1,3\}$, and
$op(\pi')=\{2\}, cl(\pi')=\{1,3,5\}, tr(\pi)=\{4\}$.
\end{example}

\subsection{Nonnesting partitions}

Let $\Phi$ be the root system of type $\Psi$, with $\Psi=B_n$ or $D_n$.
The root poset of  $\Phi$ is the set $\Phi^+$ of positive roots partially ordered by letting $\alpha\leq\beta$ if $\beta-\alpha$ is in the positive integer span of the positive roots. A {\it nonnesting} partition of type $\Psi$ is just an antichain in $\Phi^+$, i.e. a subset of  $\Phi^+$ consisting of pairwise incomparable elements. We denote by $NN^{\Psi}(n)$ the set of nonnesting $\Psi$--partitions.

A diagram representing a nonnesting $B_n$--partition $\pi$ can be drawn over the ground set $\overline{n}<\ldots<\overline{2}<\overline{1}<0<1<2<\ldots<n$ as follows: for $i,j\in[n]$, we include an arc between $i$ and $j$, and between $\overline{i}$ and $\overline{j}$, if $\pi$ contains the root $e_i-e_j$; an arc between $i$ and $\overline{j}$, and between $\overline{i}$ and $j$ if $\pi$ contains the root $e_i+e_j$; and arcs between $i$ and $0$ and $\overline{i}$ and $0$ is $\pi$ contains the root $e_i$.
The presence of $0$ in the ground set for nonnesting $B_n$--partitions is necessary to correctly represent (when it is eventually present)
the arc between a positive number $i$ an its negative (see \cite{atha}). The chains of successive arcs in the diagram become the
blocks of a $B_n$--partition, after dropping $0$, which is the partition we associate to $\pi$.
This map defines a bijection between nonnesting $B_n$--partitions of $W^{B_n}$ and $B_n$--partitions
whose diagrams, in the above sense, contain no two arcs nested one within the other. We call this diagram the nonnesting graphical representation of $\pi$, to distinguish it from the graphical representation
of the $B_n$--partition associated to $\pi$.

\begin{example}
Below is represented the nonnesting $B_5$-partition $\{\{\pm4,\pm5\},\{\overline{2},3\},\{2,\overline{3}\}\}$:
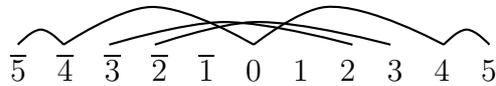
\begin{center}
\begin{pspicture}(-3,-0.5)(3,1)
\rput*[l]{0}(-3,-0.3){$\overline{5}\quad \overline{4}\quad \overline{3}\quad \overline{2}\quad \overline{1}
\quad 0\quad 1\quad 2\quad 3\quad 4\quad 5$}
\pscurve[]{}(-1.7,0)(-0.2,0.3)(1.5,0)\pscurve[]{}(-1.1,0)(0.2,0.3)(2,0)\pscurve[]{}(-2.9,0)(-2.6,0.2)(-2.3,0)
\pscurve[]{}(-2.3,0)(-1.15,0.5)(0.2,0)\pscurve[]{}(3.3,0)(3,0.2)(2.7,0)\pscurve[]{}(2.7,0)(1.15,0.5)(0.2,0)
\end{pspicture}
\end{center}
\end{example}

The positive roots of $D_n$ are those of $B_n$ other than $e_i$, $1\leq  i \leq n$. The same
rules as before determine the diagram of a nonnesting $D_n$--partition relatively to the ground set
$\overline{n}<\ldots<\overline{3}<\overline{2}<\overline{1},1<2<\ldots<n$, in which the integers $1$ and $\overline{1}$ are not comparable. This means that, for instance, an arc with 1 as endpoint and another one with $\overline{1}$ as endpoint are not considered nested. The chains of successive arcs in the diagram become the blocks of a $D_n$--partition, whose diagrams, in the above sense, contain no two arcs nested one within the other. The zero block, if it is eventually present, contains $1$ if a nesting $\{e_i - e_1, e_i +e_1\}$ appears in $\pi$.

\begin{example}The $D_5$-partition $\{\{\overline{4},1,2\},\{4,\overline{1},\overline{2}\},\{3,5\},\{\overline{3},4\}\}$, represented below is nonnesting:
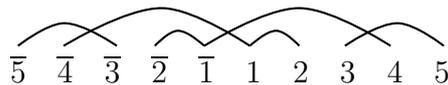
\begin{center}
\begin{pspicture}(-3,-0.5)(3,1)
\rput*[l]{0}(-3,-0.3){$\overline{5}\quad \overline{4}\quad \overline{3}\quad \overline{2}\quad \overline{1}
 \quad 1\quad 2\quad 3\quad 4\quad 5$}
\pscurve[]{}(-1.6,0)(-2.3,0.3)(-2.9,0)\pscurve[]{}(1.4,0)(2.1,0.3)(2.7,0)
\pscurve[]{}(-2.3,0)(-1,0.5)(0.15,0)\pscurve[]{}(0.15,0)(0.5,0.2)(0.8,0)
\pscurve[]{}(-1.1,0)(-0.8,0.2)(-0.45,0)\pscurve[]{}(2,0)(0.8,0.5)(-0.45,0)
\end{pspicture}
\end{center}
\end{example}

\section{A bijection between nonnesting and noncrossing $D_n$--partitions}

The main constituent of our bijection between the sets $NN^D(n)$ and $NC^D(n)$ is the bijection $$f:NN^B(n)\longrightarrow NC^B(n)$$
between nonnesting and noncrossing $B_n$--partitions presented in \cite{m}, which preserves openers, closers and transients. This map reduces to a bijection between noncrossing and nonnestin partitions of type $A$, and it was also used to obtain the analog bijection for type $C$, see \cite{m1}. Recently, Martin Rubey and Christian Stump \cite{stump1} have also obtained the same bijection using a different construction, and proved that this is (essentially) the unique bijection between nonnesting and noncrossing partitions of type $B$ that preserves openers, closers and transients.

\begin{theorem}[\cite{m,m1,stump1}] \label{t1}
The map $f:NN^B(n)\longrightarrow NC^B(n)$ is a bijection and it preserves openers, closers and transients.
\end{theorem}

The following Lemma is a consequence of the definition of $f$ and Theorem \ref{t1}.

\begin{lemma} \label{l2}
Let $\pi\in NN^B(n)$ and $i\neq \pm 1$. Then
\begin{enumerate}
\item $\pi$ has a zero block if and only if $f(\pi)$ has a zero block;
\item the zero block of $\pi$ contains $1,i$ if and only if the zero block of $f(\pi)$ contains $1,j$ for some $j\neq \pm 1$.
\end{enumerate}
\end{lemma}
\begin{proof}
$(1)$ and the fact that 1 is in the zero block of $\pi$ if and only if it is in the zero block of $f(\pi)$  follow immediately from the definition of $f$ given in \cite{m,m1}. It is easy to see that if $1,i$ belong to the zero block of $\pi$ then $1$ must necessarily be a transient, and since $f$ preserves openers, closers and transients, then $1$ must also be a transient in the zero block of $f(\pi)$. Since the greatest element of the zero block is a closer, there is at least one other positive integer $j$ in that zero block.
\end{proof}


We now point out the relation between the two posets $NN^D(n)$ and $NN^B(n)$. \\
The only $B_n$ partitions $\pi$ in $NN^D(n)$ which are not in $NN^B(n)$ are those having a nesting between an arc with endpoint 1 and another with endpoint $\overline{1}$. There are only two kinds of $D_n$--partitions in $NN^D(n)$ for which this can happen:
\begin{enumerate}
\item $D_n$--partitions $\pi$ with no zero block and having a block $B$ with entries $\overline{j},1,i$ in a row, where $1<i<j$;
\item $D_n$--partitions with zero block having integers $1,k$ in a row such that $1<k$, and with at least one block with integers
$\overline{i},j$ in a row, where $1<i<j<k$.
\end{enumerate}
All other $D_n$--partition in $NN^D(n)$ are also in $NN^B(n)$.

\begin{example}\label{nnDmenosnnB}
To exemplify the first case, consider the nonnesting $D_7$-partition \\
$\pi=\{\{\overline{6},1,4,7\},\{6,\overline{1},\overline{4},\overline{7}\},\{\overline{3},2\},\{3,\overline{2}\}\}$, represented below, noticing that $\pi$ is not in $NN^B(7)$, since the arc linking $\overline{1}$ and $6$ nest the arc linking $1$ and $4$.

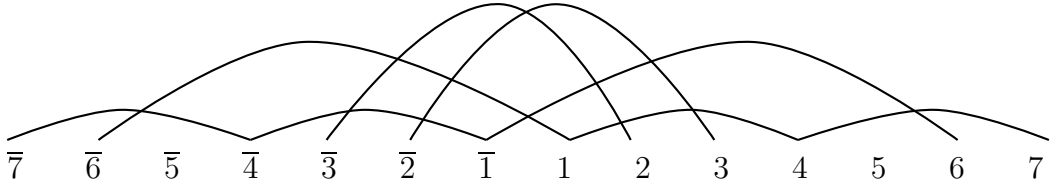
\begin{center}
\begin{pspicture}(-7,-0.5)(6,2.5)
\rput*[Bl]{0}(-7,-0.3){$\overline{7}\qquad\overline{6}\qquad\overline{5}\qquad \overline{4}\qquad \overline{3}\qquad \overline{2}\qquad \overline{1}
 \qquad 1\qquad 2\qquad 3\qquad 4\qquad 5\qquad 6\qquad 7$}
\pscurve[]{}(-7,0.2)(-5.5,0.6)(-3.8,0.2)
\pscurve[]{}(-3.8,0.2)(-2.3,0.6)(-0.7,0.2)
\pscurve[]{}(-0.7,0.2)(2.8,1.5)(5.5,0.2)

\pscurve[]{}(6.7,0.2)(5.2,0.6)(3.4,0.2)
\pscurve[]{}(3.4,0.2)(2,0.6)(0.4,0.2)
\pscurve[]{}(0.4,0.2)(-3.1,1.5)(-5.8,0.2)

\pscurve[]{}(-1.7,0.2)(0.2,2)(2.3,0.2)
\pscurve[]{}(1.2,0.2)(-0.5,2)(-2.8,0.2)
\end{pspicture}
\end{center}

An example for the remaining possibility for a $B_7$-partition to be in $NN^D(7)$ but not in $NN^B(7)$ is
$\eta=\{\{\pm1,\pm7\},\{\overline{2},5\},\{2,\overline{5}\},\{\overline{3},4\},\{3,\overline{4\}}\}$:

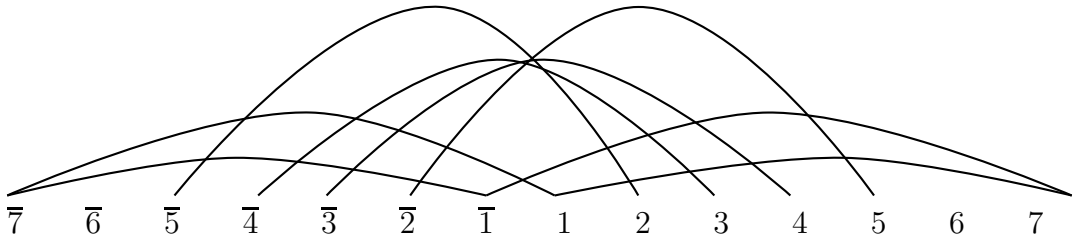
\begin{center}
\begin{pspicture}(-7,-0.5)(6,3)
\rput*[Bl]{0}(-7,-0.3){$\overline{7}\qquad\overline{6}\qquad\overline{5}\qquad \overline{4}\qquad \overline{3}\qquad \overline{2}\qquad \overline{1}
 \qquad 1\qquad 2\qquad 3\qquad 4\qquad 5\qquad 6\qquad 7$}
\pscurve[]{}(-7,0.2)(-4,0.7)(-0.7,0.2)
\pscurve[]{}(-7,0.2)(-3,1.3)(0.2,0.2)

\pscurve[]{}(7,0.2)(4,0.7)(0.2,0.2)
\pscurve[]{}(7,0.2)(3,1.3)(-0.7,0.2)

\pscurve[]{}(-1.7,0.2)(1.3,2.7)(4.4,0.2)
\pscurve[]{}(3.3,0.2)(0,2)(-2.8,0.2)

\pscurve[]{}(1.3,0.2)(-1.3,2.7)(-4.8,0.2)
\pscurve[]{}(-3.7,0.2)(-0.5,2)(2.3,0.2)
\end{pspicture}
\end{center}
The $B_7$--partition $\eta$ is not an element of $NN^B(7)$ since in the nonnesting graphical representation of $\eta$, the arc linking 0 and 1 is nested by the arc linking $\overline{2}$ and 5.

\end{example}

Given a block $B$ of a $B_n$--partition $\pi$, we write the elements of $B$ by increasing order, and $B=B^n\cup B^p$, where $B^n$, respectively $B^p$, is the subset of $B$ formed by all its negative, respectively positive, integers. Finally, we denote by $s(B)$ the least positive integer of $B$ and by
$w(B)=s(-B^{n})$, i.e. the absolute value of the largest negative entry of $B$; obviously, if $B$ has only positive integers then $s(B)$ is the opener of this block.

\begin{definition} \label{D:1}
We now design a map $\iota$ of $NN^D(n)$ into $NN^B(n)$, for any $n \geq 2$.
As one would expect, $\iota$ is just the identity for all partitions $\pi \in NN^D(n) \cap NN^B(n)$,
but otherwise a very careful attention is required.

Construct the map
$$\iota:NN^D(n)\longrightarrow NN^B(n)$$
as a three step definition as follows. Let $\pi=\{B_{0},B_1,-B_1,B_2,-B_2,\ldots,B_m,-B_m\}\in NN^D(n)$, where $B_{0}$ is the zero block.

\begin{enumerate}
\item If $B_{0}=\emptyset$, $1 \in B_{1}$ and $1  < s(B_{1} \setminus \{1\})< w(B_{1})$,
then set
$$\iota(\pi):=\{C_0,C_1,-C_1,B_2,-B_2,\ldots,B_m,-B_m\},$$ where
$C_0=\pm B_1^p\setminus\{1\}$ and $C_1=-B_1^n\cup\{1\}$.
\item If $1 \in B_{0}$ and
$B_{\ell}$ for $\ell=1,\ldots,s \leq m$, are the blocks having both positive and negative entries with
$1 < w(B_{s}) < \cdots < w(B_{1}) < s(B_{1}) < \cdots < s(B_{s})$, then set
$$\iota(\pi)=\{C_{0},C_1,-C_1,\ldots,C_s,-C_s,B_{s+1},-B_{s+1},\ldots B_m,-B_m\},$$
where
\begin{align*}
C_{0}&=\pm B_s^p\\
C_1&=-B_{s-1}^p\cup B_{0}^{p}\\
C_2&=B_2^n\cup -B_1^n\\
C_3&=B_3^n\cup B_1^p\\
&\vdots\\
C_{s-1}&=B_{s-1}^n\cup B_{s-3}^p\text{ and}\\
C_{s}&=B_s^n\cup B_{s-2}^p.
\end{align*}
\item Finally, set $\iota(\pi)=\pi$ if $\pi$ is not of type $(1)$ nor $(2)$, (viz. $\pi \in NN^D(n) \cap NN^B(n)$).
\end{enumerate}
\end{definition}

\begin{example}
The image, and respective graphical representation, of the nonnesting $D_7$-partitions $\pi$ and $\eta$, considered in example \ref{nnDmenosnnB}, by the map $\iota$ are, respectively:
$\iota(\pi)=\{\{\pm 4,\pm7\},\{1,6\},\{\overline{1},\overline{6}\},\{\overline{3},2\},\{3,\overline{2}\}\}:$

\begin{center}
\begin{pspicture}(-8,-0.5)(6,2.5)
\rput*[Bl]{0}(-8,-0.3){$\overline{7}\qquad\overline{6}\qquad\overline{5}\qquad \overline{4}\qquad \overline{3}\qquad \overline{2}\qquad \overline{1}\qquad 0
 \qquad 1\qquad 2\qquad 3\qquad 4\qquad 5\qquad 6\qquad 7$}
\pscurve[]{}(-8,0.2)(-6.5,0.6)(-4.8,0.2)\pscurve[]{}(-4.8,0.2)(-2.3,0.7)(-0.65,0.2)

\pscurve[]{}(6.5,0.2)(5.1,0.6)(3.4,0.2)\pscurve[]{}(3.4,0.2)(1.5,0.7)(-0.65,0.2)

\pscurve[]{}(-2.7,0.2)(-0.2,2.3)(2.3,0.2)
\pscurve[]{}(1.3,0.2)(-1,2.3)(-3.8,0.2)

\pscurve[]{}(-6.8,0.2)(-4,2)(-1.7,0.2)
\pscurve[]{}(5.5,0.2)(3.2,2)(0.4,0.2)
\end{pspicture}
\end{center}

\noindent and
$\iota(\eta)=\{\{\pm5\},\{\overline{4},1,7\},\{4,\overline{1},\overline{7}\},\{\overline{2},3\},\{2,\overline{3\}}\}:$

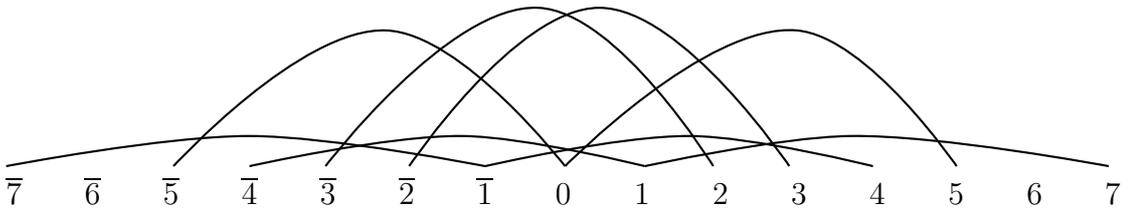
\begin{center}
\begin{pspicture}(-8,-0.5)(6,2.5)
\rput*[Bl]{0}(-8,-0.3){$\overline{7}\qquad\overline{6}\qquad\overline{5}\qquad \overline{4}\qquad \overline{3}\qquad \overline{2}\qquad \overline{1}\qquad 0
 \qquad 1\qquad 2\qquad 3\qquad 4\qquad 5\qquad 6\qquad 7$}
\pscurve[]{}(-8,0.2)(-4.8,0.6)(-1.7,0.2)
\pscurve[]{}(3.4,0.2)(1,0.6)(-1.7,0.2)

\pscurve[]{}(6.5,0.2)(3.1,0.6)(0.4,0.2)
\pscurve[]{}(0.4,0.2)(-2,0.6)(-4.8,0.2)

\pscurve[]{}(-2.7,0.2)(-0.2,2.3)(2.3,0.2)
\pscurve[]{}(1.3,0.2)(-1,2.3)(-3.8,0.2)

\pscurve[]{}(-5.8,0.2)(-3,2)(-0.65,0.2)
\pscurve[]{}(4.5,0.2)(2.4,2)(-0.65,0.2)
\end{pspicture}
\end{center}
\end{example}

\begin{theorem}\label{t2}
The map $\iota:NN^D(n)\longrightarrow NN^B(n)$ is an injection.
Moreover, the set $\iota(NN^D(n))$ is made by the nonnesting $B_n$--partitions having zero blocks, when they are present, containing $1$ and at least one more positive integer, and by the nonnesting $B_n$--partitions having a zero block without 1 and such that 1 belongs to a block $B$ for which either
$B\setminus\{1\}$ contains only positive integers or it contains both positive and negative integers.
\end{theorem}
\begin{proof}
Given $\pi\in NN^D(n)$, consider its image $\iota(\pi)$. If $\pi$ is of type $(1)$, then  $\{C_0,C_1,-C_1\}$ is clearly a nonnesting $B_n$--partition, and therefore the whole $\iota(\pi)$ is also a nonnesting $B_n$--partition. Note that $\iota(\pi)$ has a zero block without 1, and this integer is the opener of the block $C_1$, which contains only positive integers and it has cardinality strictly greater than $1$.

If $\pi$ is of type $(2)$ then the situation is similar, since by its construction, the set $\{C_{0},C_1,-C_1,\ldots,C_s,-C_s\}$ is a nonnesting
$B_n$--partition. Since $\pi\in NN^D(n)$ it follows that the whole $\iota(\pi)$ is a nonnesting $B_n$--partition
having zero block not containing 1, and this integer is a transient of the block $C_1$. It is immediate by definition that $C_1\setminus\{1\}$ contains  both positive and negative integers.

Moreover, it is not hard to see that $\iota$ is one--to--one. In fact, going backwards, for any element $\pi^{\prime}$ of $NN^B(n)$ having a zero block without 1 and such that 1 belongs to a block $B$ for which either $B\setminus\{1\}$ contains only positive integers or it contains both positive and negative integers, the inverse of the above construction for the cases $(1)$ and $(2)$ gives unambiguously a nonnesting $D_n$--partition, whose image by $\iota$ is exactly $\pi^{\prime}$. Therefore $\iota:NN^D(n)\longrightarrow NN^B(n)$ is an injection.
\end{proof}

The following result is a consequence of the construction of the map $\iota$.

\begin{corollary} \label{iotap}
The map $\iota:NN^D(n)\longrightarrow NN^B(n)$ preserves openers, closers and transients, except for the nonnesting $D_n$--partitions of type $(1)$ in Definition \ref{D:1}, for which we have
$$\begin{cases}
op(\iota(\pi)) &= op(\pi)\bigcup\{1\} \\
cl(\iota(\pi)) &= cl(\pi) \\
tr(\iota(\pi)) &= tr(\pi)\setminus\{1\}
\end{cases}.$$\qed
\end{corollary}

\begin{corollary}
The set $\left(f\circ\iota\right)\left(NN^D(n)\right) \subset NC^B(n)$ is made by the noncrossing $B_n$--partitions whose zero blocks, when they are present, contain $1$ and at least one more  positive integer, and by the noncrossing $B_n$--partitions having a zero block without 1 and such that 1 belongs to a block $B$ for which either
$B\setminus\{1\}$ contains only positive integers or it contains both positive and negative integers.
\end{corollary}
\begin{proof}
By Lemma \ref{l2}, the map $f$ transforms nonnesting $B_n$--partition with zero block having 1 and at least one other positive element into noncrossing $B_n$--partition having the same properties, and therefore also nonnesting $B_n$--partition with zero block without 1 into noncrossing $B_n$--partition with zero block without 1. Since $f$ preserves openers, closers and transients, the desired result follows.
\end{proof}

To fully present the bijection between nonnesting and noncrossing $D_n$--partitions we need to explicitly investigate the relation between the two posets
$\left(f\circ\iota\right)\left(NN^D(n)\right)$, which is a subset of $NC^B(n)$, and $NC^D(n)$. A careful examination at the notion of noncrossing $B_n$ and $D_n$--partitions shows that
if $\pi\in NC^B(n)$ but $\pi\notin NC^D(n)$ then only one of the following 3 mutually exclusive possibilities can occur:
\begin{enumerate}
\item $\pi$ has a nonzero block which contains only the integers $\pm 1$,
\item $\pi$ has a zero block which does not contain 1,
\item $\pi$ has two blocks $B$ and $B^{\prime}$ such that the first one contains integers $1,i$ in a row and the second one contains integers $\overline{k},j$, with either $\overline{k}<1<i<j$ or $\overline{k}<i<1<j$.
\end{enumerate}
Note that in the last case, notwithstanding the name \emph{noncrossing}, when considering the corresponding $(2n-2)$--gon the convex hulls $\rho(B)$
and $\rho(B^{\prime})$ actually cross.

We now design a map $\xi$ of $\left(f\circ\iota\right)\left(NN^D(n)\right)$ into $NC^D(n)$, for any $n \geq 2$.
As one would expect, $\xi$ is just the identity for all partitions
$\pi \in \left(f\circ\iota\right)\left(NN^D(n)\right) \cap NC^D(n)$, but otherwise a very careful attention is required.

\begin{definition} \label{D:3}
Construct the map
$$\xi:f\circ\iota(NN^D(n))\longrightarrow NC^D(n)$$
as a six step definition as follows. Let $\pi=\{B_{0},B_1,-B_1,\ldots,B_m,-B_m\}\in f\circ\iota(NN^D(n))$, whereth $B_{0}$ is the zero block.

\begin{enumerate}
\item
If $1 \notin B_{0} \neq \emptyset$, $1 \in B_{1}$ and $B_1 \setminus \{1\}$ is not empty and it contains both positive and negative integers, then set
$$\xi(\pi)=\{C_{0},C_1,-C_1,B_2,-B_2,\ldots,B_m,-B_m\},$$
where
$C_{0}=B_{0}\cup\{\pm 1\}$ and $C_1=B_1\setminus\{1\}$.

\item
If $B_{0}=\emptyset$, $1 \in B_{1}$ and $B_1 \setminus \{1\}$ is not empty and it contains both positive and negative integers,
and there is at least one more block containing both positive and negative integers, then assume $B_2$ is the block with the largest number $s(B)$ among all blocks containing both positive and negative integers. Set
$$\xi(\pi)=\{C_1,-C_1,C_2,-C_2,B_3,-B_3\ldots,B_m,-B_m\},$$
where
$C_1=B_1\setminus\{1\}$ and $C_2=B_2\cup\{1\}$.

\item
Suppose $B_{0}=\emptyset$, $1 \in B_{1}$ and $B_1 \setminus \{1\}$ is not empty and it contains only positive integers, and there is at least one block having both positive and negative integers. Let $B_2,B_3,\ldots,B_k,-B_k,\ldots,-B_3,-B_2$ be the collection of all blocks having both positive and negative integers by increasing order of their $s(B)$ numbers. Then set
$$\xi(\pi)=\{C_1,-C_1,\ldots,C_k,-C_k,B_{k+1},-B_{k+1}\ldots,B_m,-B_m\},$$
where
\begin{align*}
C_1&=B_1\setminus\{1\}\cup B_2^n\\
C_2&=B_2^p\cup B_3^n\\
&\vdots\\
C_{k-1}&=B_{k-1}^p\cup B_k^n\text{ and}\\
C_k&=B_k^p\cup\{1\}.
\end{align*}

\item
Similarly to case $(3)$, assume $B_{0}=\emptyset$,
$1 \in B_{1}$ and $B_1 \setminus \{1\}$ is not empty and it contains only negative integers, and there is at least one block having both positive and negative integers. Let $B_2,B_3,\ldots,B_k,-B_k,\ldots,-B_3,-B_2$ be the collection of all blocks having both positive and negative integers by increasing order of their $s(B)$ numbers. Then set
$$\xi(\pi)=\{C_1,-C_1,\ldots,C_k,-C_k,B_{k+1},-B_{k+1}\ldots,B_m,-B_m\},$$
where
\begin{align*}
C_1&=B_1\setminus\{1\}\cup B_2^p\\
C_2&=B_2^n\cup B_3^p\\
&\vdots\\
C_{k-1}&=B_{k-1}^n\cup B_k^p\text{ and}\\
C_k&=B_k^n\cup\{1\}.
\end{align*}

\item
Assume $1 \notin B_{0} \neq \emptyset$, $1 \in B_{1}$ and the block $B_1$ contains only positive integers, and there is at least one block having both positive and negative integers. Let
$B_2,B_3,\ldots,B_k,-B_k,\ldots,-B_3,-B_2$ be the collection of all blocks having both positive and negative integers by increasing order of their $s(B)$ numbers. Then set
$$\xi(\pi)=\{C_1,-C_1,\ldots,C_k,-C_k,B_{k+1},-B_{k+1}\ldots,B_m,-B_m\},$$
where
\begin{align*}
C_1&=B_1\setminus\{1\}\cup B_2^n\\
C_2&=B_2^p\cup B_3^n\\
&\vdots\\
C_{k-1}&=B_{k-1}^p\cup B_k^n\text{ and}\\
C_k&=B_k^p\cup\{-1\}\cup B_0^n.
\end{align*}

\item
Finally, set $\xi(\pi)=\pi$ otherwise, i.e., if $\pi\in\left(f\circ\iota\right)\left(NN^D(n)\right) \cap NC^D(n)$.
\end{enumerate}
\end{definition}

\begin{example}See below instances of the application of the map $\xi$ to the first five cases of definition \ref{D:3}.

\noindent Case 1.

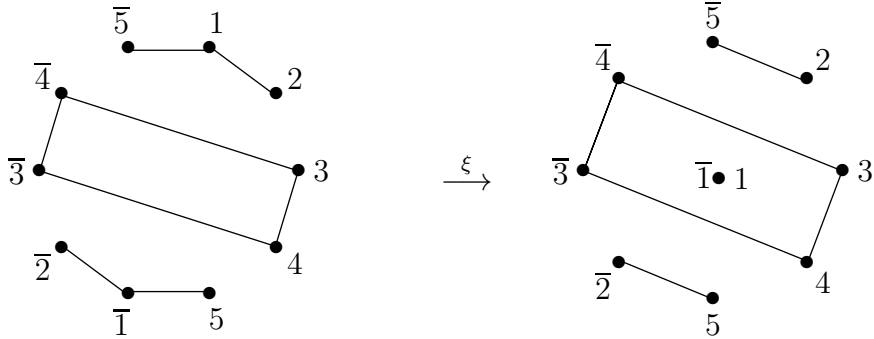
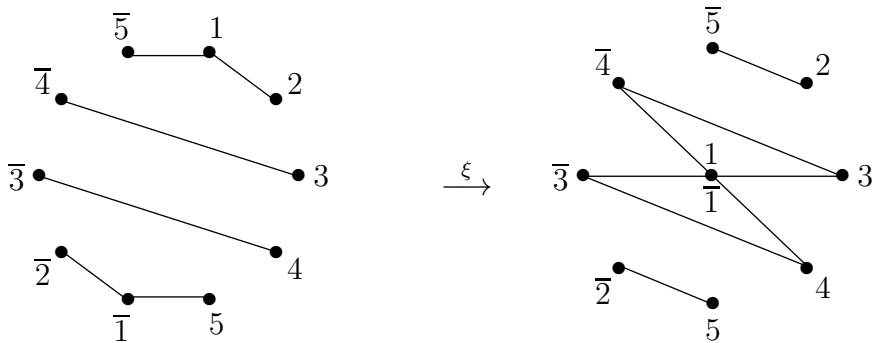
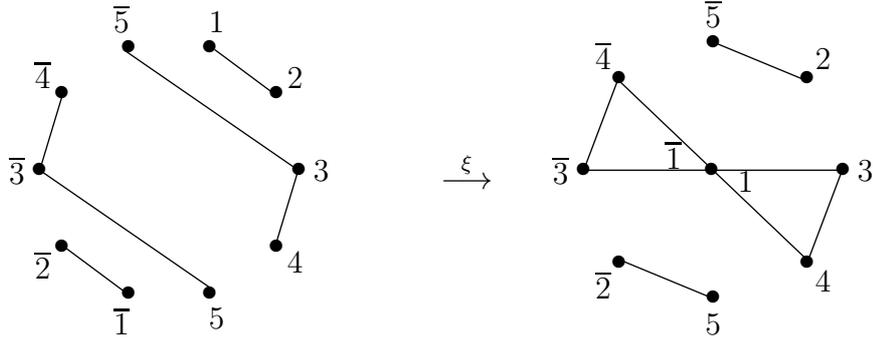
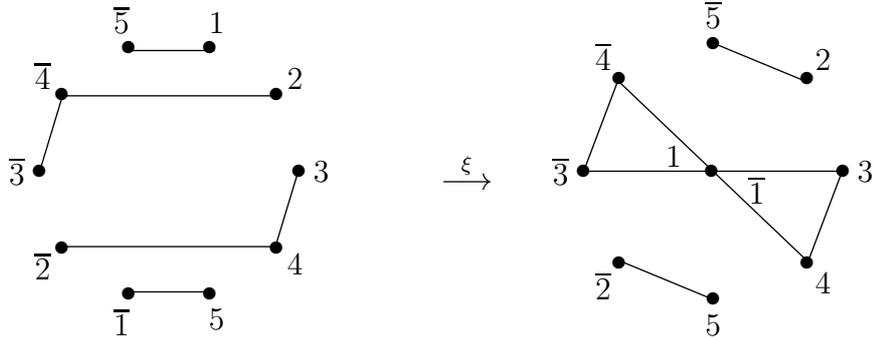
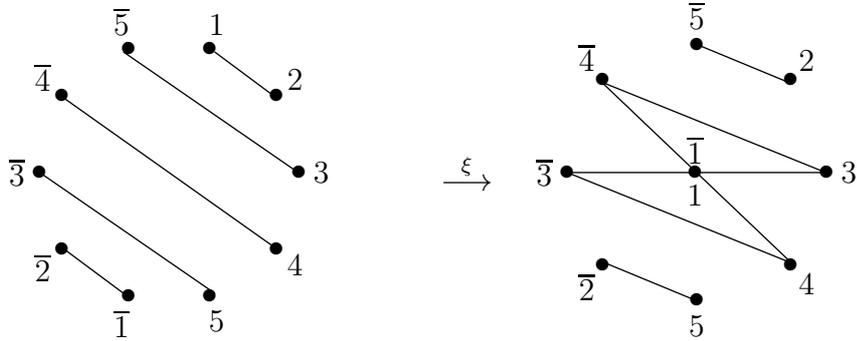
\begin{tabular}{cc}
\begin{pspicture}(-3.4,-2.5)(3.4,2.5)
 \degrees[1]
 \multido{\n=0.0+.1}{10}{%
 \uput{1.6}[\n](0,0){$\bullet$}}
 \multido{\n=0.2+-0.1,\i=1+1}{5}{%
 \uput{1.9}[\n](0,0){\i}}
 \multido{\n=0.7+-0.1,\i=1+1}{5}{%
 \uput{1.9}[\n](0,0){$\overline{\i}$}}
 %
 \uput[0](3.4,0){$\overset{\xi}{\longrightarrow}$}
 \psline[linewidth=0.5pt]{}(-1.7,0)(-1.4,1)(1.7,0)(1.4,-1)(-1.7,0)
 \psline[linewidth=0.5pt]{}(-0.6,1.6)(0.6,1.6)(1.4,1)
\psline[linewidth=0.5pt]{}(0.6,-1.6)(-0.6,-1.6)(-1.4,-1)
 \end{pspicture}
&
 \begin{pspicture}(-3.4,-2.5)(3.4,2.5)
 \degrees[1]
 \multido{\n=0.0+.125}{8}{%
 \uput{1.6}[\n](0,0){$\bullet$}}
 \multido{\n=0.125+-0.125,\i=2+1}{4}{%
 \uput{1.9}[\n](0,0){\i}}
 \multido{\n=0.625+-0.125,\i=2+1}{4}{%
 \uput{1.9}[\n](0,0){$\overline{\i}$}}
 \uput[90](-0.4,-0.05){$\overline{1}$}\uput[90](0.1,-0.1){$1$}
 \uput[0](-0.2,-0.1){$\bullet$}
 \psline[linewidth=0.5pt]{}(-1.25,1.2)(-1.7,0)(-1.25,1.2)(1.7,0)(1.25,-1.2)(-1.7,0)
 \psline[linewidth=0.5pt]{}(0,1.7)(1.2,1.2)
 \psline[linewidth=0.5pt]{}(0,-1.7)(-1.2,-1.2)
 \end{pspicture}
 \end{tabular}

\noindent Case 2.

\begin{tabular}{cc}
\begin{pspicture}(-3.4,-2.5)(3.4,2.5)
 \degrees[1]
 \multido{\n=0.0+.1}{10}{%
 \uput{1.6}[\n](0,0){$\bullet$}}
 \multido{\n=0.2+-0.1,\i=1+1}{5}{%
 \uput{1.9}[\n](0,0){\i}}
 \multido{\n=0.7+-0.1,\i=1+1}{5}{%
 \uput{1.9}[\n](0,0){$\overline{\i}$}}
 %
 \uput[0](3.4,0){$\overset{\xi}{\longrightarrow}$}
 \psline[linewidth=0.5pt]{}(-1.7,0)(1.4,-1)
 \psline[linewidth=0.5pt]{}(-1.4,1)(1.7,0)
 \psline[linewidth=0.5pt]{}(-0.6,1.6)(0.6,1.6)(1.4,1)
\psline[linewidth=0.5pt]{}(0.6,-1.6)(-0.6,-1.6)(-1.4,-1)
 \end{pspicture}
&
 \begin{pspicture}(-3.4,-2.5)(3.4,2.5)
 \degrees[1]
 \multido{\n=0.0+.125}{8}{%
 \uput{1.6}[\n](0,0){$\bullet$}}
 \multido{\n=0.125+-0.125,\i=2+1}{4}{%
 \uput{1.9}[\n](0,0){\i}}
 \multido{\n=0.625+-0.125,\i=2+1}{4}{%
 \uput{1.9}[\n](0,0){$\overline{\i}$}}
 \uput[90](-0.3,-0.3){$\overline{1}$}\uput[90](-0.3,0.3){$1$}
 \uput[0](-0.3,0){$\bullet$}
 \psline[linewidth=0.5pt]{}(-1.25,1.2)(1.7,0)(0,0)(-1.25,1.2)
 \psline[linewidth=0.5pt]{}(1.25,-1.2)(-1.7,0)(0,0)(1.25,-1.2)
 \psline[linewidth=0.5pt]{}(0,1.7)(1.2,1.2)
 \psline[linewidth=0.5pt]{}(0,-1.7)(-1.2,-1.2)
 \end{pspicture}
 \end{tabular}

\noindent Case 3.

\begin{tabular}{cc}
\begin{pspicture}(-3.4,-2.5)(3.4,2.5)
 \degrees[1]
 \multido{\n=0.0+.1}{10}{%
 \uput{1.6}[\n](0,0){$\bullet$}}
 \multido{\n=0.2+-0.1,\i=1+1}{5}{%
 \uput{1.9}[\n](0,0){\i}}
 \multido{\n=0.7+-0.1,\i=1+1}{5}{%
 \uput{1.9}[\n](0,0){$\overline{\i}$}}
 %
 \uput[0](3.4,0){$\overset{\xi}{\longrightarrow}$}
 \psline[linewidth=0.5pt]{}(-0.6,1.6)(1.7,0)(1.4,-1)
 \psline[linewidth=0.5pt]{}(0.6,1.6)(1.4,1)
  \psline[linewidth=0.5pt]{}(0.6,-1.6)(-1.7,0)(-1.4,1)
 \psline[linewidth=0.5pt]{}(-0.6,-1.6)(-1.4,-1)
 \end{pspicture}
&
 \begin{pspicture}(-3.4,-2.5)(3.4,2.5)
 \degrees[1]
 \multido{\n=0.0+.125}{8}{%
 \uput{1.6}[\n](0,0){$\bullet$}}
 \multido{\n=0.125+-0.125,\i=2+1}{4}{%
 \uput{1.9}[\n](0,0){\i}}
 \multido{\n=0.625+-0.125,\i=2+1}{4}{%
 \uput{1.9}[\n](0,0){$\overline{\i}$}}
 \uput[90](-0.8,0.2){$\overline{1}$}\uput[90](0.15,-0.15){$1$}
 \uput[0](-0.3,0){$\bullet$}
 \psline[linewidth=0.5pt]{}(-1.25,1.2)(-1.7,0)(0,0)(-1.25,1.2)
 \psline[linewidth=0.5pt]{}(1.25,-1.2)(1.7,0)(0,0)(1.25,-1.2)
 \psline[linewidth=0.5pt]{}(0,1.7)(1.2,1.2)
 \psline[linewidth=0.5pt]{}(0,-1.7)(-1.2,-1.2)
 \end{pspicture}
 \end{tabular}

\noindent Case 4.

\begin{tabular}{cc}
\begin{pspicture}(-3.4,-2.5)(3.4,2.5)
 \degrees[1]
 \multido{\n=0.0+.1}{10}{%
 \uput{1.6}[\n](0,0){$\bullet$}}
 \multido{\n=0.2+-0.1,\i=1+1}{5}{%
 \uput{1.9}[\n](0,0){\i}}
 \multido{\n=0.7+-0.1,\i=1+1}{5}{%
 \uput{1.9}[\n](0,0){$\overline{\i}$}}
 %
 \uput[0](3.4,0){$\overset{\xi}{\longrightarrow}$}
 \psline[linewidth=0.5pt]{}(-0.6,1.6)(0.6,1.6)%
  \psline[linewidth=0.5pt]{}(-1.7,0)(-1.4,1)(1.4,1)
 \psline[linewidth=0.5pt]{}(0.6,-1.6)(-0.6,-1.6)%
  \psline[linewidth=0.5pt]{}(1.7,0)(1.4,-1)(-1.4,-1)
 \end{pspicture}
&
 \begin{pspicture}(-3.4,-2.5)(3.4,2.5)
 \degrees[1]
 \multido{\n=0.0+.125}{8}{%
 \uput{1.6}[\n](0,0){$\bullet$}}
 \multido{\n=0.125+-0.125,\i=2+1}{4}{%
 \uput{1.9}[\n](0,0){\i}}
 \multido{\n=0.625+-0.125,\i=2+1}{4}{%
 \uput{1.9}[\n](0,0){$\overline{\i}$}}
 \uput[90](-0.8,0.2){$1$}\uput[90](0.3,-0.25){$\overline{1}$}
 \uput[0](-0.3,0){$\bullet$}
 \psline[linewidth=0.5pt]{}(-1.25,1.2)(-1.7,0)(0,0)(-1.25,1.2)
 \psline[linewidth=0.5pt]{}(1.25,-1.2)(1.7,0)(0,0)(1.25,-1.2)
 \psline[linewidth=0.5pt]{}(0,1.7)(1.2,1.2)
 \psline[linewidth=0.5pt]{}(0,-1.7)(-1.2,-1.2)
 \end{pspicture}
 \end{tabular}

\noindent Case 5.

\begin{tabular}{cc}
\begin{pspicture}(-3.4,-2.5)(3.4,2.5)
 \degrees[1]
 \multido{\n=0.0+.1}{10}{%
 \uput{1.6}[\n](0,0){$\bullet$}}
 \multido{\n=0.2+-0.1,\i=1+1}{5}{%
 \uput{1.9}[\n](0,0){\i}}
 \multido{\n=0.7+-0.1,\i=1+1}{5}{%
 \uput{1.9}[\n](0,0){$\overline{\i}$}}
 %
 \uput[0](3.4,0){$\overset{\xi}{\longrightarrow}$}
 \psline[linewidth=0.5pt]{}(-0.6,1.6)(1.7,0)
 \psline[linewidth=0.5pt]{}(0.6,1.6)(1.4,1)
 \psline[linewidth=0.5pt]{}(1.4,-1)(-1.4,1)
  \psline[linewidth=0.5pt]{}(0.6,-1.6)(-1.7,0)
 \psline[linewidth=0.5pt]{}(-0.6,-1.6)(-1.4,-1)
 \end{pspicture}

 \begin{pspicture}(-3.4,-2.5)(3.4,2.5)
 \degrees[1]
 \multido{\n=0.0+.125}{8}{%
 \uput{1.6}[\n](0,0){$\bullet$}}
 \multido{\n=0.125+-0.125,\i=2+1}{4}{%
 \uput{1.9}[\n](0,0){\i}}
 \multido{\n=0.625+-0.125,\i=2+1}{4}{%
 \uput{1.9}[\n](0,0){$\overline{\i}$}}
 \uput[90](-0.3,-0.3){$1$}\uput[90](-0.3,0.3){$\overline{1}$}
 \uput[0](-0.3,0){$\bullet$}
 \psline[linewidth=0.5pt]{}(-1.25,1.2)(1.7,0)(0,0)(-1.25,1.2)
 \psline[linewidth=0.5pt]{}(1.25,-1.2)(-1.7,0)(0,0)(1.25,-1.2)
 \psline[linewidth=0.5pt]{}(0,1.7)(1.2,1.2)
 \psline[linewidth=0.5pt]{}(0,-1.7)(-1.2,-1.2)
 \end{pspicture}
 \end{tabular}
\end{example}

\begin{theorem}\label{t3}
The map $\xi:\left(f\circ\iota\right)\left(NN^D(n)\right)\longrightarrow NC^D(n)$ is a bijection.
\end{theorem}
\begin{proof}
By definition \ref{D:3} it is easy to see that $\xi(\pi)$ is always a noncrossing $D_n$--partition.

By the investigation of the two posets $\left(f\circ\iota\right)\left(NN^D(n)\right)$ and $NC^D(n)$ immediately preceding the definition of $\xi$, it follows that the noncrossing $B_n$--partitions which satisfy one of the the first five cases in Definition \ref{D:3} are the only ones which are not in $NC^D(n)$. Furthermore, the images by $\xi$ of all these $B_n$--partitions are not in $NC^B(n)$, since in all cases the block containing the integer 1 always crosses at least one block containing both positive and negative integers.

Therefore $\xi$ restricted to the first five cases of Definition \ref{D:3} sends
$\left(f\circ\iota\right)\left(NN^D(n)\right) \setminus NC^D(n)$
(recall $\left(f\circ\iota\right)\left(NN^D(n)\right) \subset NC^B(n)$)
into $NC^D(n) \setminus NC^B(n)$.

Finally, the construction can be easily reversed. The only partitions in $NC^D(n)$ which are not in $NC^B(n)$ are those such that the block $B$ containing  1 contains also some other integer $j$ which crosses some other block containing both positive and negative integers. It follows that if
$\pi^{\prime}\in NC^D(n) \setminus NC^B(n)$ then $\pi^{\prime}$ must belong to the set
$\left(\xi\circ f\circ\iota\right)\left(NN^D(n)\right)$, showing that $\xi$ is indeed bijective.
\end{proof}

The composition of the maps $\iota$, $f$ and $\xi$ gives the desired bijection between the nonnesting and noncrossing $D_n$--partitions.

\begin{theorem}
The map $\xi\circ f\circ\iota:NN^D(n)\longrightarrow NC^D(n)$ is a bijection, and furthermore, writing
\begin{eqnarray*}
NN^D(n) & = & \left(NN^D(n) \bigcap NN^B(n)\right) \biguplus \left(NN^D(n) \setminus NN^B(n)\right), \\
NC^D(n) & = & \left(NC^D(n) \bigcap NC^B(n)\right) \biguplus \left(NC^D(n) \setminus NC^B(n)\right),
\end{eqnarray*}
we have that the two restrictions
\begin{eqnarray*}
\xi\circ f\circ\iota &:& NN^D(n) \bigcap NN^B(n) \longrightarrow NC^D(n) \bigcap NC^B(n) \\
\xi\circ f\circ\iota &:& NN^D(n) \setminus NN^B(n) \longrightarrow NC^D(n) \setminus NC^B(n)
\end{eqnarray*}
are also bijections.

Moreover the map $\xi\circ f\circ\iota$ preserves openers, closers and transients.
\end{theorem}
\begin{proof}
The bijectivity follows from Theorems \ref{t1}, \ref{t2}, \ref{t3}, and the proof of the latter Theorem.

By Definition \ref{D:3} the only case where openers, closers and transients are not preserved by the map $\xi$ is when a $B_n$--partition has a nonzero block without 1, and 1 belongs to a block containing only positive integers, i.e. in the case $(5)$ of the definition. In this case we have
$1\in op(\pi)$ but $1\in tr(\xi(\pi))$.

For all other integers their properties of being openers, closers or transients are preserved, i.e.
$$\begin{cases}
op(\xi(\pi)) &= op(\pi)\setminus\{1\} \\
cl(\xi(\pi)) &= cl(\pi) \\
tr(\xi(\pi)) &= tr(\pi)\cup\{1\}
\end{cases}.$$

It is easy to see that $\pi$ is a partition satisfying case $(5)$ of Definition \ref{D:3} if and only if it is the image of a partition
satisfying case $(1)$ of Definition \ref{D:1}. Therefore recalling Corollary \ref{iotap}, the desired result follows.
\end{proof}

\end{document}